\documentclass [reqno,10pt,oneside]{amsart}

\pdfoutput=1
\usepackage{hyperref}
\usepackage{amsthm}
\usepackage{amscd}
\usepackage{amsmath,amssymb}
\usepackage{algorithm}
\usepackage[noend]{algorithmic}

\usepackage{dcolumn}
\newcolumntype{d}{D{.}{.}{-1}}

\theoremstyle {plain}
\newtheorem {thm}{Theorem}[section]
\newtheorem {prop}[thm]{Proposition}
\newtheorem {lem}[thm]{Lemma}

\theoremstyle {definition}
\newtheorem {defn}[thm]{Definition}

\theoremstyle {remark}
\newtheorem {rem}[thm]{Remark}

\newtheorem {exmp}[thm]{Example}

\DeclareMathOperator{\Hom}{Hom}

\DeclareMathOperator{\Ann}{Ann}

\DeclareMathOperator{\LM}{LM}

\DeclareMathOperator{\Ker}{Ker}

\DeclareMathOperator{\Spec}{Spec}
\DeclareMathOperator{\Sing}{Sing}
\DeclareMathOperator{\Quot}{Q}

\newcommand{\Q}{{\mathbb Q}}

\newcommand{\gen}[1]{\left\langle #1 \right\rangle}

\begin{document}

\bibliographystyle{alpha}

\title{Parallel Algorithms for Normalization}

\author{Janko B\"ohm}
\address{Janko B\"ohm\\ Department of Mathematics\\ University of Kaiserslautern\\
Erwin-\linebreak Schr\"odinger-Str.\\ 67663 Kaiserslautern\\ Germany}
\email{boehm@mathematik.uni-kl.de}

\author{Wolfram Decker}
\address{Wolfram Decker\\ Department of Mathematics\\ University of Kaiserslautern\\
Erwin-Schr\"odinger-Str.\\ 67663 Kaiserslautern\\ Germany}
\email{decker@mathematik.uni-kl.de}

\author{Gerhard Pfister}
\address{Gerhard Pfister\\ Department of Mathematics\\ University of Kaiserslautern\\
Erwin-Schr\"odinger-Str.\\ 67663 Kaiserslautern\\ Germany}
\email{pfister@mathematik.uni-kl.de}

\author{Santiago Laplagne}
\address{Santiago Laplagne\\ Departamento de Matem\'atica\\ FCEN\\ Universidad de 
Buenos Aires - Ciudad Universitaria\\ Pabell\'on I - (C1428EGA) - Buenos Aires\\ Argentina}
\email{slaplagn@dm.uba.ar}

\author{Andreas Steenpa{\ss}}
\address{Andreas Steenpa{\ss}\\ Department of Mathematics\\ University of Kaiserslautern\\ 
Erwin-Schr\"odinger-Str.\\ 67663 Kaiserslautern\\ Germany}
\email{steenpass@mathematik.uni-kl.de}

\author{Stefan Steidel}
\address{Stefan Steidel\\ Department of Mathematics\\ University of Kaiserslautern\\ 
Erwin-Schr\"odinger-Str.\\ 67663 Kaiserslautern\\ Germany}
\email{steidel@mathematik.uni-kl.de}


\date{\today}

\maketitle

\begin{abstract}
Given a reduced affine algebra $A$ over a perfect field $K$, we
present parallel algorithms to compute the normalization $\overline{A}$
of $A$. Our starting point is the algorithm of
\cite{GLS}, which is an improvement of de Jong's algorithm, see
\cite{deJong98, DGPJ}. First, we propose to stratify the singular
locus $\Sing(A)$ in a way which is compatible with normalization,
apply a local version of the normalization algorithm at each stratum,
and find $\overline{A}$ by putting the local results together.
Second, in the case where $K=\Q$ is the field of rationals,
we propose modular versions of the global and local--to--global algorithms.
We have implemented our algorithms in the computer algebra
system {\sc{Singular}} and compare their performance with that of
the algorithm of \cite{GLS}.
In the case where $K=\Q$, we also discuss the
use of modular computations of Gr\"obner bases, radicals, and primary
decompositions.
We point out that in most examples, the new algorithms
outperform the algorithm of \cite{GLS} by far, even if we do not run
them in parallel.
\end{abstract}

\section{Introduction} \label{secIntro}
 
Normalization is an important concept in commutative algebra, with
applications in algebraic geometry and singularity theory. We are interested
in computing the normalization $\overline{A}$ of a reduced affine $K$--algebra
$A$, where $K$ is a perfect field. For this, a number of algorithms have been
proposed, but not all of them are of practical interest
(see the historical account in \cite{GLS}). A milestone is de Jong's
algorithm, see \cite{deJong98, DGPJ}, which is based on the normality criterion of
\cite{GR}, and which has been implemented in \textsc{Singular}
(see \cite{DGPS}), \textsc{Macaulay2} (see \cite{GS}), and \textsc{{Magma}} 
(see \cite{BCP}). The algorithm of \cite{GLS} (\emph{GLS normalization
algorithm} for short), which is also based on the Grauert and Remmert
criterion, is an improvement of de Jong's algorithm. It is implemented in
\textsc{{Singular}}. The algorithm proposed by \cite{LP} and \cite{SS} is
designed for the characteristic $p$ case. It is implemented in
\textsc{{Singular}} and \textsc{{Macaulay2}} and works well for small $p$.

In view of modern multi--core computers, the parallelization of fundamental
algorithms becomes increasingly important. Our objective in this paper is to
present parallel versions of the GLS normalization algorithm in that we reduce
the general problem to computational problems which are easier and do not
depend on each other. It turns out that in most cases, the new algorithms
outperform the GLS algorithm by far, even if we do not run them in parallel.

We start in Section \ref{sect:glob-norm-alg} by reviewing the basic ideas of
the GLS algorithm. In particular, we recall the normality criterion of Grauert
and Remmert. In Section \ref{sect:norm-via-loc}, we present a local version of
the normality criterion which applies to a stratification of the singular
locus $\Sing(A)$ of $A$. This allows us to find $\overline{A}$ by a
local--to--global approach. Section \ref{secModMeth} contains a discussion of
modular methods for the GLS algorithm and its local--to--global version.
Timings are presented in Section \ref{sect:timings}.

\section{The GLS Normalization Algorithm} \label{sect:glob-norm-alg}

Referring to \cite{GLS} and \cite{GP} for details and proofs, we sketch the
GLS normalization algorithm. We begin with some general remarks. For these,
$A$ may be any reduced Noetherian ring.

\begin{defn}
Let $A$ be a reduced Noetherian ring. The {\emph{normalization}} of $A$,
written $\overline{A}$, is the integral closure of $A$ in its total ring of
fractions $\Quot(A)$. We call $A$ {\emph{normal}} if $A =\overline{A}$.
\end{defn}

We write
\[
\Spec(A)=\{P\subseteq A\mid P{\text{ prime ideal}}\}
\]
for the {\emph{spectrum}} of $A$ and $V(J)=\{P\in\Spec(A)\mid P\supseteq J\}$
for the {\emph{vanishing locus}} of an ideal $J$ of $A$. If $P\in
\operatorname{Spec}(A)$, then $A_{P}$ denotes the localization of $A$ at $P$.
More generally, if $S$ is a multiplicatively closed subset of $A$ and $M$ is
an $A$--module, then $S^{-1}M$ denotes the localization of $M$ at $S$. 

\vspace{0.1cm}

Taking into account that normality is a local property, we call
\[
N(A)=\{P\in\Spec(A)\mid A_{P}\mbox{ is not normal}\}
\]
the {\emph{non--normal locus}} of $A$. Furthermore, we write%
\[
\Sing(A)=\{P\in\Spec(A)\mid A_{P}\mbox{ is not regular}\}
\]
for the {\emph{singular locus}} of $A$. Then $N(A)\subseteq\Sing(A)$.

\begin{rem}
\label{rem:nnlocus-singlocus-dim-one} A Noetherian local ring of dimension one
is normal if and only if it is regular. See \cite[Theorem 4.4.9]{JP}.
\end{rem}

\begin{defn}
Let $A$ be a reduced Noetherian ring. The {\emph{conductor}} of $A$ in
$\overline{A}$ is the ideal
\[
\mathcal{C}_{A}= \Ann_{A}(\overline{A}/A) = \{ a\in A \mid a \overline{A}
\subseteq A\}.
\]

\end{defn}


\begin{lem}
\label{lemma:role-of-cond} Let $A$ be a reduced Noetherian ring. Then $N(A)
\subseteq V(\mathcal{C}_{A})$. Furthermore, $\overline{A}$ is module--finite
over $A$ if and only if $\mathcal{C}_{A}$ contains a non--zerodivisor of $A$.
In this case, $N(A) = V(\mathcal{C}_{A})$.
\end{lem}

To state the aforementioned Grauert and Remmert criterion, we need:

\begin{lem}
\label{lemma:prep-GR} Let $A$ be a reduced Noetherian ring, and let
$J\subseteq A$ be an ideal containing a non--zerodivisor $g$ of $A$. Then the
following hold:

\begin{enumerate}
\item If $\varphi\in\Hom_{A}(J, J)$, then the fraction
$\varphi(g)/g\in\overline{A}$ is independent of the choice of $g$, and
$\varphi$ is multiplication by $\varphi(g)/g$.

\item There are natural inclusions of rings
\[
A\subseteq\Hom_{A}(J, J)\cong\frac{1}{g} \, (gJ :_{A} J)\subseteq\overline
{A}\subseteq\Quot(A),\; a \mapsto\varphi_{a}, \; \varphi\mapsto\frac
{\varphi(g)}{g},
\]
where $\varphi_{a}: J \rightarrow J$ denotes the multiplication by $a \in A$.
\end{enumerate}
\end{lem}

\begin{prop}
[\cite{GR}]\label{prop:testideal} \label{prop:crit-GR} Let $A$ be a reduced
Noe\-therian ring, and let $J \subseteq A$ be an ideal satisfying the
following conditions:

\begin{enumerate}
\item $J$ contains a non--zerodivisor $g$ of $A$,

\item $J$ is a radical ideal,

\item $V(\mathcal{C}_{A}) \subseteq V(J)$.
\end{enumerate}

\noindent Then $A$ is normal iff $A\cong\Hom_{A}(J,J)$ via the map which sends
$a$ to $\varphi_{a}$.
\end{prop}

\begin{defn}
\label{def:test-pair} A pair $(J,g)$ as in Proposition \ref{prop:testideal} is
called a {\emph{test pair}} for $A$, and $J$ is called a {\emph{test ideal}}
for $A$.
\end{defn}

By Lemma \ref{lemma:role-of-cond}, test pairs exist iff $\overline{A}$ is
module--finite over $A$. Given such a pair $(J,g)$, the idea of finding
$\overline{A}$ is to successively enlarge $A$ until the normality criterion
allows us to stop (since $A$ is Noetherian, this will eventually happen in the
module--finite case). Starting from $A_{0}=A$, we get a chain of extensions of
reduced Noetherian rings
\[
A = A_{0} \subseteq\dots\subseteq A_{i-1} \subseteq A_{i} \subseteq
\dots\subseteq A_{m} = \overline{A}.
\]
Here, $A_{i+1} = \Hom_{A_{i}}(J_{i},J_{i}) \cong\frac1g (gJ_{i} :_{A_{i}}
J_{i})$, where $J_{i}$ is the radical of the extended ideal $J A_{i}$, for
$i\geq1$.
Note that ($J_{i},g)$ is indeed a test pair for $A_{i}$:

\begin{rem}[\cite{GLS}, Prop.\ 3.2]\label{rem:ext-test-ideal} Let $A$ be a reduced
Noetherian ring such that $\overline{A}$ is module--finite over $A$, and let
$A\subseteq A^{\prime}\subseteq\overline{A}$ be an intermediate ring. Clearly,
every non--zerodivisor $g\in A$ of $A$ is a non--zerodivisor of $\Quot(A)$. In
particular, it is a non--zerodivisor of $A^{\prime}$. Furthermore, if
$\mathcal{C}_{A^{\prime}}$ is the conductor of $A^{\prime}$ in $\overline
{A^{\prime}}=\overline{A}$, then $\mathcal{C}_{A^{\prime}}\supseteq
\mathcal{C}_{A}$. It follows that every prime ideal $Q\in N(A^{\prime
})=V(\mathcal{C}_{A^{\prime}})$ contracts to a prime ideal $P=Q\cap A\in
N(A)=V(\mathcal{C}_{A})$. Hence, if $(J,g)$ is a test pair for $A$, then
$P\supseteq J$, which implies that $Q\supseteq\sqrt{JA^{\prime}}=:J^{\prime}$.
We conclude that $(J^{\prime},g)$ is a test pair for $A^{\prime}$.
\end{rem}

Explicit computations rely on explicit representations of the $A_{i}$ as
$A$--algebras. These will be obtained as an application of Lemma
\ref{lemma:k-algebra-structure} below. To formulate the lemma, we use the
following notation. Let $J\subseteq A$ be an ideal containing a
non--zerodivisor $g$ of $A$, and let
$A$--module generators $u_{0} = g, u_{1}, \dots, u_{s}\;\!$ for $gJ :_{A} J$
be given. Choose variables $T_{1}, \dots, T_{s}$, and consider the
epimorphism
\[
\Phi: A[T_{1}, \dots, T_{s}] \rightarrow\frac{1}{g} \, (gJ :_{A}J), \; T_{i}
\mapsto\frac{u_{i}}{g}.
\]
The kernel of $\Phi$ describes the $A$--algebra relations on the $u_{i}/g$. We
single out two types of relations:

\begin{itemize}
\item Each $A$--module syzygy
\[
\alpha_{0}u_{0} + \alpha_{1}u_{1} + \ldots+ \alpha_{s} u_{s} = 0, \quad
\alpha_{i} \in A,
\]
gives an element $\alpha_{0} +\alpha_{1}T_{1} + \ldots+ \alpha_{s}T_{s}%
\in\Ker \Phi$, which we call a \textit{linear relation}.

\item Developing each product $\frac{u_{i\vphantom{j}}}{g}\frac{u_{j}}{g}$, $1
\leq i \leq j \leq s$, as a sum $\frac{u_{i\vphantom{j}}}{g}\frac{u_{j}}{g} =
\sum_{k} \beta_{ijk} \frac{u_{k}}{g}$, we get elements $T_{i}T_{j} - \sum_{k}
\beta_{ijk} T_{k}$ in $\Ker \Phi$, which we call \textit{quadratic relations}.
\end{itemize}

\noindent It is easy to see that these linear and quadratic relations already
generate $\Ker\Phi$. We thus have:

\begin{lem}
\label{lemma:k-algebra-structure} Let $A$ be a reduced Noetherian ring, and
let $J\subseteq A$ be an ideal containing a non--zerodivisor $g$ of $A$. Then,
given $A$--module generators $u_{0} = g, u_{1}, \dots, u_{s}\;\!$ for $gJ
:_{A} J$, we have an isomorphism of $A$--algebras
\[
A[T_{1},\dots,T_{s}]/R\cong\frac{1}{g} \, (gJ :_{A}J), \; T_{i} \mapsto
\frac{u_{i}}{g},
\]
where $R$ is the ideal generated by the linear and quadratic relations
described above.

\end{lem}

The following result from \cite{GLS} will allow us to find the normalization
in a way such that all calculations except the computation of the radicals
$\sqrt{J_{i}}$ can be carried through in the original ring $A$:

\begin{thm}
\label{thm:comp-in-orig-ring} Let $A$ be a reduced Noetherian ring, let
$J\subseteq A$ be an ideal containing a non--zerodivisor $g$ of $A$, let
$A\subseteq A^{\prime}\subseteq\Quot(A)$ be an intermediate ring such that
$A^{\prime}$ is module--finite over $A$, and let $J^{\prime}=\sqrt{JA^{\prime
}}$.
Let $U$ and $H$ be ideals of $A$ and $d\in A$ such that $A^{\prime}=\frac
{1}{d}U$ and $J^{\prime}=\frac{1}{d}H$, respectively. Then
\[
(gJ^{\prime}:_{A^{\prime}}J^{\prime})=\frac{1}{d}\,(dgH:_{A}H)\subseteq
\Quot(A).
\]

\end{thm}


\begin{rem}
In the case where $A=K[X_{1}\dots,X_{n}]/I$ is a reduced affine algebra over a
field $K$, let $P_{1},\dots,P_{r}$ be the associated primes of the radical
ideal $I$. Then
\[
\overline{A}\cong\overline{K[X_{1},\dots,X_{n}]/P_{1}}\times\cdots
\times\overline{K[X_{1},\dots,X_{n}]/P_{r}},
\]
and $\overline{A}$ is module--finite over $A$ by Emmy Noether's finiteness
theorem (see \linebreak \cite{SH}). Thus, using techniques for primary decomposition as
in \cite[Remark 4.6]{GLS}, the computation of normalization can be reduced to
the case where $A$ is an affine domain (that is, $I$ is a prime ideal). When
writing our algorithms in pseudocode, we will always start from a domain $A$.
Talking about a non--zerodivisor then just means to talk about a non--zero element.
\end{rem}

\begin{rem}
\label{remJacobian} If $A$ is an affine domain over a perfect field $K$, we
can apply the Jacobian criterion (see \cite{Eis}): If $M$ is the Jacobian
ideal\footnote{The {\emph{Jacobian ideal}} of $A$ is generated by the images
of the $c\times c$ minors of the Jacobian matrix $\big(\frac{\partial f_{i}%
}{\partial x_{j}}\big)$, where $c$ is the codimension, and $f_{1},\dots,f_{r}$
are polynomial generators for $I$.} of $A$, then $M$ is non--zero and
contained in the conductor $\mathcal{C}_{A}$ (see \cite[Lemma 4.1]{GLS}).
Hence, we may choose $\sqrt{M}$ together with any non--zero element $g$ of
$\sqrt{M}$ as an initial test pair. Implementing all this, the GLS
normalization algorithm will find an ideal $U\subseteq A$ and a denominator
$d\in\mathcal{C}_{A}$ such that
\[
\overline{A}=\frac{1}{d}\,U\subseteq\Quot(A).
\]
Since $M$ is contained in $\mathcal{C}_{A}$, any non--zero element of $M$ is
valid as a denominator: If $0\neq c\in M$, then $c\cdot\frac{1}{d}U=:U'$ is an
ideal of $A$, so that $\frac{1}{d}U=\frac{1}{c}U'$.
\end{rem}

For the purpose of comparison with the local approach of the next section, we
illustrate the GLS algorithm by an example:

\begin{exmp}
\label{ex a3 E6} For
\[
A=K[x,y]=K[X,Y]/\langle X^{4}+Y^{2}(Y-1)^{3}\rangle,
\]
the radical of the Jacobian ideal is
\[
J:=\left\langle x,y\left(  y-1\right)  \right\rangle _{A},
\]
and we can take $g:=x\in J$ as a non--zerodivisor of $A$. In its first step,
starting with the initial test pair $(J,x)$, the normalization algorithm
produces the following data:
\[
U^{(1)}:=xJ\;:_{A}\;J=\left\langle x,y(y-1)^{2}\right\rangle _{A}%
\;\,\text{and}\;\,A_{1}:=A[t_{1}]:=A[T_{1}]/I_{1}\cong\frac{1}{x}U^{(1)},
\]
with relations and isomorphism given by
\[
I_{1}=\left\langle -T_{1}x+y(y-1)^{2},T_{1}y(y-1)+x^{3},T_{1}^{2}%
+x^{2}(y-1)\right\rangle _{A[T_{1}]}%
\]
and
\[
t_{1}\mapsto\frac{y(y-1)^{2}}{x},
\]
respectively. In the next step we find
\begin{align*}
J_{1}:=  &  \;\sqrt{\left\langle x,y(y-1)\right\rangle _{A_{1}}}=\left\langle
x,y(y-1),t_{1}\right\rangle _{A_{1}}\\
=  &  \;\frac{1}{x}\left\langle x^{2},xy(y-1),y(y-1)^{2}\right\rangle
_{A}=:\frac{1}{x}H_{1}.
\end{align*}
Using the test pair $(J_{1},x)$ and applying Theorem
\ref{thm:comp-in-orig-ring} and Lemma \ref{lemma:k-algebra-structure}, we get%
\begin{align*}
\frac{1}{x}(xJ_{1}\;:_{A_{1}}\;J_{1})  &  =\frac{1}{x^{2}}(x^{2}H_{1}%
\;:_{A}\;H_{1})\\
&  =\frac{1}{x^{2}}\left\langle x^{2},xy(y-1),y(y-1)^{2}\right\rangle
_{A}=:\frac{1}{x^{2}}U^{(2)}%
\end{align*}
and%
\[
A_{2}:=A[t_{2},t_{3}]:=A[T_{2},T_{3}]/I_{2}\cong\frac{1}{x^{2}}U^{(2)},
\]
with relations and isomorphism given by%
\begin{align*}
I_{2}  &  =\left\langle T_{2}x-T_{3}(y-1),-T_{3}x+y(y-1),T_{2}y(y-1)+x^{2}%
,T_{2}y^{2}(y-1)^{2}+T_{3}x^{3},\right. \\
&  \qquad\left.  T_{2}^{2}+(y-1),T_{2}T_{3}+x,T_{3}^{2}-T_{2}y\right\rangle
\end{align*}
and
\[
t_{2}\mapsto\frac{y(y-1)^{2}}{x^{2}},\;t_{3}\mapsto\frac{y(y-1)}{x},\;
\]
respectively. In the final step, we find that $A_{2}$ is normal, so that
$\overline{A} = A_{2}$.
\end{exmp}

\section{Normalization via Localization} \label{sect:norm-via-loc}

In this section, we discuss a local--to--global approach for computing 
normalization. Our starting point is the following result:

\begin{prop}
\label{prop:local-to-global-I} Let $A$ be a reduced Noetherian ring. Suppose
that the singular locus $\Sing(A)=\{P_{1},\dots,P_{s}\}$ is finite. For
$i=1,\dots,s$, let $S_{i}=A\setminus P_{i}$, and let an intermediate
ring $A\subset A^{(i)}\subset\overline{A}$ be given such that $S_{i}%
^{-1}A^{(i)}=\overline{S_{i}^{-1}A}$. Then
\[
\sum_{i=1}^{s}A^{(i)}=\overline{A}.
\]

\end{prop}

\begin{proof}
We will show a more general result in Proposition
\ref{prop:local-to-global-II} below.
\end{proof}

That $\Sing(A)$ is finite is, for example, true if $A$ is the coordinate ring
of a curve. Whenever $\Sing(A)=\{P_{1},\dots,P_{s}\}$ is finite, the
proposition allows us to find $\overline{A}$ by normalizing locally at each
$P_{i}$ using Proposition \ref{prop local grauert-remmert} below, and putting
the local results together. In the case where $\Sing(A)$ is not finite,
working just with the (finitely many)  minimal primes in $\Sing(A)$ will
not give the correct result. However, it is still possible to obtain
$\overline{A}$ as a finite sum of local contributions: The idea is to stratify
$\Sing(A)$ in a way which is compatible with normalization. For this, if
$P\in\Sing(A)$, set
\[
L_{P}=\bigcap_{P\supseteq\widetilde{P}\in\Sing(A)}\widetilde{P}.
\]
We stratify $\Sing(A)$ according to the values of the function $P\mapsto
L_{P}$. That is, if%
\[
\mathcal{L}=\{L_{P}\mid P\in\Sing(A)\}
\]
denotes the set of all possible values, then the strata are the
sets%
\[
V_{L}=\{P\in\Sing(A)\mid L_{P}=L\}\text{, }L\in\mathcal{L}\text{.}%
\]
We write $\operatorname{Strata}(A)=\left\{  V_{L}\mid L\in\mathcal{L}\right\}
$ for the set of all  strata. If $P_{1},...,P_{r}$ denote the minimal
primes in $\Sing(A)$, we have
\[
\mathcal{L}\subseteq\left\{  {%
{\textstyle\bigcap\nolimits_{i\in\Gamma}}
}P_{i}\mid\Gamma\subseteq\{1,...,r\}\right\}.
\]
Hence, the set of strata is finite. By construction, the singular locus is
the disjoint union of all strata. For $V\in\operatorname{Strata}(A)$, write
$L_{V}$ for the constant value of $P\mapsto L_{P}$ on $V$.

We can now state and prove a result which is more general than Proposition
\ref{prop:local-to-global-I}:

\begin{prop}
\label{prop:local-to-global-II} Let $A$ be a reduced Noetherian ring with
stratification of the singular locus $\operatorname{Strata}(A) =\{V_{1}%
,...,V_{s}\}$. For $i=1,\dots,s$, let an intermediate ring
$A\subseteq A^{(i)}\subseteq\overline{A}$ be given such that $S^{-1}%
A^{(i)}=\overline{S^{-1}A}$ for each $S=A\setminus P$, $P\in V_{i}$. Then
\[
\sum_{i=1}^{s}A^{(i)}=\overline{A}\text{.}%
\]

\end{prop}

\begin{proof}
By construction, $B:=\sum_{i=1}^{s}A^{(i)}\subseteq\overline{A}$. We wish to
show equality. It suffices to show that if $P\in\Spec(A)$ is a prime ideal and
$S=A\setminus P$, then $S^{-1}B=S^{-1}\overline{A}$. If $P\in\Sing(A)$, then
$P\in V_{i}$ for some $i$. Hence, $S^{-1}A^{(i)}=\overline{S^{-1}A}$,
and the local equality is obtained from the chain of inclusions
\[
S^{-1}A^{(i)}\subseteq S^{-1}B\subseteq S^{-1}\overline{A}=\overline{S^{-1}%
A}\text{.}%
\]
If $P\not \in\Sing(A)$, then $S^{-1}A$ is normal, and the local equality follows 
likewise from the chain of inclusions
\[
S^{-1}A\subseteq S^{-1}B\subseteq S^{-1}\overline{A}=\overline{S^{-1}%
A}\text{.}%
\]
\end{proof}

For a given stratum $V=V_i$, the modification of the Grauert and Remmert criterion
below will allow us to find a ring $A^{(i)}$ as above along the lines of the previous section:

\begin{prop}
\label{prop local grauert-remmert} Let $A$ be a reduced Noetherian ring such
that $\overline{A}$ is module--finite over $A$, and let $A\subseteq A^{\prime
}\subseteq\overline{A}$ be an intermediate ring. Let $V\in
\operatorname{Strata}(A)$, and let $J^{\prime}=\sqrt{L_{V}A^{\prime}}$. Suppose
that $L_{V}$ contains a non--zerodivisor $g$ of $A$. If
\[
A^{\prime}\cong\operatorname*{Hom}\nolimits_{A^{\prime}}(J^{\prime},J^{\prime
})
\]
via the map which sends $a^{\prime}$ to $\varphi_{a^{\prime}}$, then the
localization $S^{-1}A^{\prime}$ with $S=A\setminus P$ is normal for each $P\in
V$.
\end{prop}

\begin{proof}
The assumption and \cite[Proposition 2.10]{Eis} give
\[
S^{-1}A^{\prime}\cong S^{-1}(\Hom_{A^{\prime}}(J^{\prime},J^{\prime}%
))\cong\Hom_{S^{-1}A^{\prime}}(S^{-1}J^{\prime},S^{-1}J^{\prime}).
\]
Hence, the result will follow from the Grauert and Remmert criterion
(Proposition \ref{prop:crit-GR}) applied to $S^{-1}A^{\prime}$ once we show
that the localized ideal $S^{-1}J^{\prime}$ satisfies the three conditions of
the criterion. First, since forming radicals commutes with localization,
$S^{-1}J^{\prime}$ is a radical ideal. Second, the image of $g$ in
$S^{-1}A^{\prime}$ is a non--zerodivisor of $S^{-1}A^{\prime}$ contained in
$S^{-1}J^{\prime}$. Third, we show that $V(\mathcal{C}_{S^{-1}A^{\prime}%
})=N(S^{-1}A^{\prime})\subseteq V(S^{-1}J^{\prime})$. For this, we first
note that
\[
V(\mathcal{C}_{S^{-1}A})=N(S^{-1}A)=\{S^{-1}\widetilde{P}\mid\widetilde{P}%
\in N(A)\text{, }\widetilde{P}\subseteq P\}.
\]
Indeed, prime ideals of $S^{-1}A$ correspond to prime ideals of $A$ contained in
$P$. Let now $Q\in N(S^{-1}A^{\prime})$. Then, as shown in Remark
\ref{rem:ext-test-ideal}, $Q$ contracts to some $S^{-1}\widetilde{P}\subseteq S^{-1}A$ 
with $\widetilde{P}\in N(A)$, $\widetilde{P}\subseteq P$. This implies that
\[
Q\supseteq\sqrt{(S^{-1}\widetilde{P})(S^{-1}A^{\prime})}=\sqrt{S^{-1}(\widetilde
{P}A^{\prime})}=S^{-1}(\sqrt{\widetilde{P}A^{\prime}})\supseteq S^{-1}J^{\prime},
\]
as desired.
\end{proof}

In the situation of Proposition \ref{prop local grauert-remmert}, let a
non--zerodivisor $g\in L_{V}$ of $A$ be known. Then, using $(L_{V},g)$ instead
of a test pair as in Definition \ref{def:test-pair}, and proceeding as in the
previous section, we get a chain of rings
\[
A\subseteq A_{1}\subseteq\dots\subseteq A_{m}\subseteq\overline{A}%
\]
such that $S^{-1}(A_{m})$ is normal and, hence, equal to $S^{-1}\overline
{A}=\overline{S^{-1}A}$ for all $S=A\setminus P$, $P\in V$.

\vspace{0.1cm}

Summing up, we are lead to Algorithms \ref{alg1} and \ref{alg2} below.

\begin{algorithm}[H]                      
\caption{Normalizing the localizations}          
\label{alg1}                           
\begin{algorithmic}[1]
\REQUIRE An affine domain $A = K[X_1,\dots , X_n]/I$ over a perfect field $K$, a stratum $V\in
\operatorname{Strata}(A)$,
and $0\neq g \in L_V$.
\ENSURE An ideal $U \subseteq A$\ and $d \in A$ with
$\frac 1d U\subseteq\overline{A}$ and  $S^{-1} (\frac 1d U) = \overline{S^{-1} A}$ for all 
$S=A\setminus P$, $P\in V$.
\RETURN the result of the GLS normalization algorithm applied to $(L_V,g)$;
\end{algorithmic}
\end{algorithm}

\begin{algorithm}[H]                      
\caption{Normalization via localization}          
\label{alg2}                           
\begin{algorithmic}[1]
\REQUIRE An affine domain $A = K[X_1,\dots , X_n]/I$ over a perfect field $K$.
\ENSURE An ideal $U \subseteq A$ and $d \in A$ such that $\overline{A}=\frac 1d U\subseteq \Quot(A)$.
\STATE $J := \sqrt{M}$, where $M$ is the Jacobian ideal of $A$;
\STATE choose  $0\neq g \in J$;
\STATE compute the strata of the singular locus $\operatorname{Strata}(A)=\{V_{1},...,V_{s}\}$;
\FORALL{$i$}
\STATE apply Algorithm \ref{alg1} to $\left(
V_{i},g\right)  $ to find an ideal $U_{i}\subseteq A$ and a power $d_i=g^{m_i}$ with
$A\subseteq \frac{1}{d_i} U_{i}\subseteq\overline{A}$ and
$S^{-1} ( \frac{1}{d_i} U_{i}) = \overline{S^{-1} A}$ for all $S=A\setminus P$, $P\in V_i$;
\ENDFOR
\STATE $m:= \max\{m_1,\dots, m_s\}$,\, $d:=g^m$,\, $U := \sum_i g^{m-m_i}U_i$;
\RETURN $(U, d)$;
\end{algorithmic}
\end{algorithm}

\begin{rem}
In Algorithm \ref{alg2}, it may be more efficient to choose possibly different
non--zero elements $g_{i}\in L_{V_{i}}$. In Step 5, the algorithm computes, then, pairs
$(U_{i}^{\prime},d_{i})$ with ideals $U_{i}^{\prime}\subseteq A$ and powers
$d_{i}=g_{i}^{m_{i}}$. As explained in \cite[Remark 4.3]{GLS}, starting from
the $(U_{i}^{\prime},d_{i})$, we may always find a denominator $d\in M$ and
ideals $U_{i}\subseteq A$ such that $\frac{1}{d}U_{i}=\frac{1}{d_{i}}%
U_{i}^{\prime}$ for all $i$. Then, the desired result is $(\sum_{i}U_{i}%
,d)$.

\end{rem}

\begin{rem}
In Algorithm \ref{alg2}, it is sufficient to consider the minimal strata,
that is, the strata $V$ such that $L_{V}$ is minimal with respect to inclusion. Denote, as
above, the minimal primes of the singular locus of $A$ by $P_{1},...,P_{r}$.
We can obtain the minimal $L_{V}$ as all possible intersections
${\textstyle\bigcap\nolimits_{i\in\Gamma}} P_{i}$, with subsets $\Gamma
\subseteq\{1,...,r\}$ which are maximal with the property that ${\textstyle\sum
\nolimits_{i\in\Gamma}} P_{i}\neq\left\langle 1\right\rangle $.
\end{rem}

\begin{exmp}
We come back to the coordinate ring $A$ of the curve $C$ with defining
polynomial $f(X,Y)=X^{4}+Y^{2}(Y-1)^{3}$ from Example \ref{ex a3 E6} to
discuss normalization via localization. The curve $C$ has a double point of
type $A_{3}$ at $\left(  0,0\right)  $ and a triple point of type $E_{6}$ at
$\left(  0,1\right)  $. We illustrate Algorithm \ref{alg2}, using for both
singular points the non--zerodivisor $g=x$: For the $A_{3}$--singularity,
consider%
\[
P_{1}=\left\langle x,y\right\rangle _{A}\;{\text{ and }}\;S_{1}=A\setminus
P_{1}\text{.}%
\]
The local normalization algorithm yields $\overline{S_{1}^{-1}A}=S_{1}%
^{-1}(\frac{1}{d_{1}}U_{1})$, where
\[
d_{1}=x^{2}\;{\text{ and }}\;U_{1}=\left\langle x^{2},\text{ }y(y-1)^{3}%
\right\rangle _{A}\text{.}%
\]
For the $E_{6}$--singularity, considering%
\[
P_{2}=\left\langle x,y-1\right\rangle _{A}\;{\text{ and }}\;S_{2}=A\setminus
P_{2}\text{,}%
\]
we get\thinspace\ $\overline{S_{2}^{-1}A}=S_{2}^{-1}(\frac{1}{d_{2}}U_{2})$,
where%
\[
d_{2}=x^{2}\;{\text{ and }}\;U_{2}=\left\langle x^{2},\text{ }xy^{2}\left(
y-1\right)  ,\text{ }y^{2}\left(  y-1\right)  ^{2}\right\rangle _{A}\text{.}%
\]
Combining the local contributions, we get%
\[
\frac{1}{d}U=\frac{1}{d_{1}}U_{1}+\frac{1}{d_{2}}U_{2}\text{,}%
\]
with $d=x^{2}$ and
\[
U=\left\langle x^{2},\text{ }xy^{2}\left(  y-1\right)  ,\text{ }%
y(y-1)^{3},\text{ }y^{2}\left(  y-1\right)  ^{2}\right\rangle _{A}\text{.}%
\]
A moment's thought shows that $U$ coincides with the ideal $U^{(2)}$ found in
Example \ref{ex a3 E6}.
\end{exmp}

The local--to--global approach is usually much faster than the global
algorithm even when not run in parallel. The reason is that the minimal primes
of the singular locus are much simpler than the singular locus itself.
Therefore, in the local--to--global case, the intermediate rings are much
easier to handle. Most notably, the representations of the intermediate rings
as affine rings involve considerably less variables than in the global case.
In the following example, we exemplify this difference.

\begin{exmp}
Consider the \emph{projective} plane curve defined by the polynomial
\[
f_{1,4}=(X^{5}+Y^{5}+Z^{5})^{2}-4(X^{5}Y^{5}+X^{5}Z^{5}+Y^{5}Z^{5}%
)\in\mathbb{Q}[X,Y,Z],
\]
which will be reconsidered in Section \ref{sect:timings} with respect to
timings. After the coordinate transformation $Z\mapsto3X-2Y+Z$, all
singularities of the projective curve lie in the affine chart $Z\neq0$. Write
\[
f=f_{1,4}(X,Y,3X-2Y+1)\in \mathbb{Q}[X,Y]=:W
\]
for the defining polynomial of the affine curve, and let $A=\mathbb{Q}[x,y]=
W/\langle f\rangle$.

The curve has $15$ singular points: the radical of the Jacobian ideal $M$
decomposes as
\begin{align*}
\sqrt{M} & =\langle y,121x^{4}+142x^{3}+64x^{2}+13x+1\rangle\cap\langle
y,2x+1\rangle\\
& \cap\langle211y^{4}-131y^{3}+51y^{2}-11y+1,3x-2y+1\rangle\\
& \cap\langle11y^{4}-23y^{3}+19y^{2}-7y+1,x\rangle\cap\langle y+1,x+1\rangle
\cap\langle3y-1,x\rangle.
\end{align*}
more complicated (see below).
We compare the global approach to the local strategy at the singularity
corresponding to the test ideal $J=\langle y,2x+1\rangle$.

In the local setting, we use the non--zerodivisor $g=y$ and compute the ideal
quotient
\begin{align*}
U_{1} & =gJ:J\\
& =\langle y,29282x^{9}+83369x^{8}+105668x^{7}+78296x^{6}+37382x^{5}%
+11926x^{4}\\
& \qquad\; +2542x^{3}+349x^{2}+28x+1\rangle.
\end{align*}
We observe that in addition to $y$, the ideal $U_{1}$ requires only one more
generator. Hence, the representation of $A_{1}\cong\frac{1}{g}(gJ:J)$ as an
affine ring $A_{1}=A[T_{1}]/I_{1}$ requires only one additional variable
$T_{1}$. The ideal $I_{1}$ is generated by $10$ linear relations and one
quadratic relation. Next, we compute the radical of the image of $J$ in
$A_{1}$. Technically, this means to compute the radical $\sqrt{J+I_{1}}$ in
the polynomial ring $W[T_{1}]$ (here, by abuse of notation, we denote the 
preimage of $J$ in the polynomial ring also by $J$). The ideal $I_{1}$ is quite
complicated. Since $J$ is generated by linear polynomials, however, the
reduced Gr\"obner basis of $J+I_{1}$ is very simple. It is easily computed as
$$J+I_{1}=\langle Y,2X+1,16384T_{1}^{2}-T_{1}+625\rangle.$$ As a consequence,
the computation of the radical $$J_{1}=\sqrt{J+I_{1}}=\langle Y,2X+1,128T_{1}
-25\rangle$$ is cheap as well. In the next step, $A_{2}$ can be represented as
an affine algebra over $A$ with, again, only one new variable. Hence,
verifying that $A_{2}$ is already a local contribution to $\overline{A}$ at
the singularity corresponding to $J$ is also cheap.

In contrast, the global approach uses the test ideal $J=\sqrt{M}$, which is
generated by one polynomial of degree $3$ and three polynomials of degree $6$.
As a non--zerodivisor, we consider the lowest degree generator $g = 3x^{2}
y-2xy^{2}+xy$. As in the local case, the first ideal quotient $gJ : J$ is
easily obtained. However, in addition to $g$, it requires three more
generators. Hence, as an affine algebra over $A$, it is represented as $A_{1}
= A[T_{1}, T_{2}, T_{3}] / I_{1}$, where the ideal of relations $I_{1}
\subseteq A[T_{1}, T_{2}, T_{3}]$ is generated by $6$ linear and $6$ quadratic
relations. No significant reduction occurs in $J + I_{1}$ since $J$ does not
contain any linear polynomial. The complexity of Buchberger's algorithm grows
doubly--exponentially in the number of variables. Compared to the local case,
this increase in complexity makes the computation of $\sqrt{J + I_{1}}$
considerably more expensive. In fact, \textsc{Singular} does not compute the
corresponding Gr\"obner basis within $2000$ seconds.
\end{exmp}

\section{Modular Methods} \label{secModMeth}

Algorithm \ref{alg2} from Section \ref{sect:norm-via-loc} is parallel in
nature since the computations of the local normalizations do not depend on
each other. In this section, we describe a modular way of parallelizing both
the GLS normalization algorithm from Section \ref{sect:glob-norm-alg} and the
local--to--global algorithm from Section \ref{sect:norm-via-loc} in the case
where $K={\mathbb{Q}}$ is the field of rationals. One possible approach is to
just replace \emph{all} involved Gr\"{o}bner basis respectively radical
computations by their modular variants as introduced by \cite{ArnoldE} 
and \cite{IPS}. These
variants are either probabilistic or require rather expensive tests to verify
the results at the end. In order to reduce the number of verification tests,
we provide a direct modularization for the normalization algorithms. The
approach we propose requires, in principle, only one verification at the end.
In the local--to--global setup, however, it is reasonable to additionally
handle the Gr\"{o}bner basis computation for the Jacobian ideal, the
subsequent primary decomposition, and the recombination of the local results
by modular techniques. We exemplarily describe the modularization of the GLS
normalization algorithm as outlined in Section \ref{sect:glob-norm-alg}. Each
of the local normalizations in Algorithm \ref{alg2} from Section
\ref{sect:norm-via-loc} can be modularized similarly.

\vspace{0.15cm}
Fix a global monomial ordering $>$ on the semigroup of
monomials in the set of variables $X=\{X_{1},\ldots,X_{n}\}$. Consider the
polynomial rings $W={\mathbb{Q}}[X]$ and, given an integer $N\geq2$,
$W_{N}=(\mathbb{Z}/N\mathbb{Z})[X]$. If $T\subseteq W$ or $T\subseteq W_{N}$
is a set of polynomials, then
denote by $\operatorname{LM}(T):=\{\operatorname{LM}(f)\mid f\in T\}$ its set
of leading monomials. If $\frac{a}{b}\in\mathbb{Q}$ with $\gcd(a,b)=1$ and
$\gcd(b,N)=1$, set $\left(  \frac{a}{b}\right)  _{N}:=(a+N\mathbb{Z}%
)(b+N\mathbb{Z})^{-1}\in\mathbb{Z}/N\mathbb{Z}$. If $f\in W$ is a polynomial
such that $N$ is coprime to any denominator of a coefficient of $f$, then
$f_{N}\in W_{N}$ is the polynomial obtained by reducing each coefficient
modulo $N$ as just described. If $H=\{h_{1},\dots,h_{t}\}$ is a Gr\"{o}bner 
basis in $W$ such that $N$ is coprime to any denominator in any $h_{i}$, 
then write $H_{N}=\{(h_{1})_{N},\dots,(h_{t})_{N}\}$.  Given an ideal 
$I\subseteq W$, set $I_{N}=\left\langle f_{N}\mid 
f\in I\cap\mathbb{Z}[X]\right\rangle \subseteq W_{N}$
and $(W/I)_{N}=W_{N}/I_{N}$.

\begin{rem}\label{rem:bad-primes-in-Farey}
For practical purposes, the ideal $I\subseteq W$ is given by a set of generators 
$f_{1},\ldots,f_{r}$. Then for all but finitely many primes $p$, the 
ideal $I_p$ can be realized via the equality
\[
I_{p}=\left\langle (f_{1})_{p},...,(f_{r})_{p}\right\rangle \subseteq
W_{p}\text{.}%
\]
When performing the modular algorithm described below, we reject a prime $p$ 
if the ideal on the right hand side is not well--defined. Otherwise, we work with this 
ideal instead of $I_p$. The finitely many primes where the two ideals differ will not 
influence the result if we apply error tolerant rational reconstruction (see Remark 
\ref{rmk error tolerant farey}).
\end{rem}

From a practical point of view, we work
with ideals of the polynomial ring $W$ containing $I$, but think of them
as ideals of the quotient ring $A=W/I$. Therefore, we simplify our notation
as follows: If $I\subseteq J \subseteq W$ are ideals, then we denote the
ideal induced by $J$ in $A$ also by $J$.  Vice versa, if $J \subseteq A$
is an ideal, then its preimage in $W$ is also denoted by $J$.
Similarly for $W_N$.

From now on, $I=\left\langle f_{1},\ldots,f_{r}\right\rangle \subseteq W$
will be a prime ideal. We wish to compute the normalization of the
affine domain $A=W/I$ using modular methods. For this, we fix a 
polynomial $d\in W$ which represents a non--zero element
in the Jacobian ideal $M$ of $A$. This element of $M$ will also be denoted 
by $d$. It will  serve as  a \textquotedblleft
universal denominator\textquotedblright\ for all normalizations 
in positive characteristic as well as  for the final normalization in characteristic zero 
(see Remark \ref{remJacobian} for the choice of denominators). In 
characteristic zero, we write $U(0)$ for the ideal of $A$ which satisfies 
$\frac{1}{d}U(0)=\overline A$, and $G(0)$ for the reduced Gr\"{o}bner 
basis\footnote{Recall that reduced Gr\"{o}bner bases
are uniquely determined. For practical purposes, however, we do not 
need to reduce the Gr\"{o}bner bases involved since the lifting process described
below only requires that the result is  uniquely
determined by the algorithm.} of $U(0)$. Furthermore, we write $V(0)\subseteq
A[T_{1},\ldots,T_{s}]$ for the ideal\footnote{With respect to ideals of
$W[T_{1},\ldots,T_{s}]$ and $A[T_{1},\ldots,T_{s}]$, we use the same setup and notation
as for ideals of $W$ and $A$.} of relations on the elements of $\frac{1}{d}G(0)$ which represents 
$\overline{A}$ as an $A$--algebra as in Lemma \ref{lemma:k-algebra-structure}.
We denote the reduced Gr\"obner basis of $V(0)$ by $R(0)$. In the same way, 
if $p$ is a prime number which does not divide any denominator in the 
reduced Gr\"{o}bner basis of $I$ and such that $A_p$ is a domain
and $d_{p}$ is non--zero and contained in the conductor\footnote{From 
a practical point of view, we check whether $d_p$ is in the Jacobian ideal 
of $A_p$.} of $A_p$, we use $U(p)$, $G(p)$, $V(p)$, and $R(p)$ in 
characteristic $p$.

Note that $G(0)_{p}$ is not necessarily equal to $G(p)$. However, as
we will show in\ Lemma \ref{lem lucky lift} below, equality holds for 
all but finitely many primes $p$. Relying on this fact, the basic idea of
the modular normalization algorithm can be described as follows. First,
compute the Jacobian ideal $M$ of $A$ and choose a polynomial $d\in W$ representing
a non--zero element $d\in M$. Second, choose a set $\mathcal{P}$ of prime numbers
at random, and compute, for each $p\in \mathcal{P}$, reduced Gr\"{o}bner 
bases $G(p)\subseteq W_{p}$ such that $\frac{1}{d_{p}}\left\langle G(p)\right\rangle
\subseteq\Quot(A_{p})$ is the normalization of $A_{p}$. Third, lift the
modular Gr\"{o}bner bases to a set of polynomials $G\subseteq W$ and define
$U:=\left\langle G\right\rangle $. We then expect that $U=U(0)$ and $G=G(0)$.

The lifting process has two steps. First, assuming that all
$\operatorname{LM}(G(p))$, $p\in\mathcal{P}$, are equal, we can lift the
Gr\"{o}bner bases in the set $\mathcal{GP}:=\{G(p)\mid p\in\mathcal{P}\}$ to a
set of polynomials $G(N)\subseteq W_{N}$, with $N:=\prod_{p\in\mathcal{P}}p$.
For this, apply the Chinese remainder algorithm to the coefficients of the
corresponding polynomials occurring in the  $G(p)$, $p\in\mathcal{P}$.
Second, compute a set of polynomials $G\subseteq W$ by lifting the modular
coefficients occurring in $G(N)$ to rational coefficients as described in 
\cite{FareyPaper}:

\begin{rem}
\label{rmk error tolerant farey} Rational reconstruction via the Chinese remainder theorem
and Gaussian reduction  is error--tolerant in the following sense: Let
$N_{1}$ and $N_{2}$ be integers with $\gcd(N_{1},N_{2})=1$, and let $\frac
{a}{b}\in\mathbb{Q}$ with $\gcd(a,b)=\gcd(N_{1},b)=1$. Set
$r_{1}:=\left(  \frac{a}{b}\right)  _{N_{1}}\in\mathbb{Z}/N_{1}\mathbb{Z}$, let
$r_{2}\in\mathbb{Z}/N_{2}\mathbb{Z}$ be arbitrary, and denote by $r$ the image of
$(r_{1},r_{2})$ under the isomorphism%
\[
\mathbb{Z}/N_{1}\mathbb{Z}\times\mathbb{Z}/N_{2}\mathbb{Z}\rightarrow
\mathbb{Z}/(N_{1}N_{2}) \mathbb{Z}\text{.}%
\]
Lifting $r$ to a rational number via Gaussian reduction will generate,
starting from $(a_{0},b_{0})=(N_1N_2,0)$ and $(a_{1},b_{1})=(r,1)$, a sequence 
of rational numbers $(a_{i},b_{i})$ obtained by setting
\[
(a_{i-2},b_{i-2})=q_{i}(a_{i-1},b_{i-1})+(a_{i},b_{i}),
\]
where $q_i$ is chosen such that $(a_i,b_i)$ has minimal Euclidean length. 
Computing this sequence until the Euclidean length does not decrease strictly 
any more, we obtain a tuple $(a_{i},b_{i})$ with $\frac{a_{i}}{b_{i}}=\frac{a}{b}$,
provided that $N_{2}\ll N_{1}$. For details, see \cite{FareyPaper}.
\end{rem}

Just as for $\mathcal{GP}$, we proceed for the set 
of reduced Gr\"obner bases $\mathcal{RP}%
:=\{R(p)\mid p\in\mathcal{P}\}$ giving the modular algebra relations.

As for other modular algorithms based on Chinese remaindering, we need
suitably adapted notions of a lucky prime and a sufficiently large set of
lucky primes:

\begin{defn}
\label{defnLucky} Using the notation introduced above, we define:
\begin{enumerate}
\item A prime number $p$ is called \emph{lucky} for $A$ if $U(0)_{p}=U(p)$,
$V(0)_{p}=V(p)$, and the following hold:

\begin{enumerate}
\item $A_{p}$ is a domain.

\item $d_{p}$ is a non--zero element in the conductor of $A_{p}$.

\item $\LM(G(0))=\LM(G(p))$.

\item $\LM(R(0))=\LM(R(p))$.
\end{enumerate}

\noindent Otherwise $p$ is called \emph{unlucky} for $A$.

\item A finite set $\mathcal{P}$ of lucky primes for $A$ is called \emph{sufficiently
large} for $A$ if
\begin{align*}
\prod_{p\in\mathcal{P}}p\geq\max\left\{2\cdot\left\vert c\right\vert ^{2}\;\bigg|\;
\begin{array}{l} \text{$c$ a denominator or numerator of a} \\ \text{coefficient occurring in $G(0)$ or  $R(0)$}
\end{array} \right\}
\end{align*}
\end{enumerate}
\end{defn}

\begin{rem}
A modular algorithm for the basic task of computing Gr\"obner bases is presented
in \cite{ArnoldE} and \cite{IPS}. In contrast to our situation here, where we wish to 
find the ideal $U(0)$ by computing its reduced Gr\"obner basis $G(0)$, Arnold's 
algorithm starts from an ideal which is already given. 
If $p$ is a prime number, $J\subset W$ is an ideal, $H(0)$ is the reduced 
Gr\"{o}bner basis of $J$, and $H(p)$ the reduced Gr\"{o}bner basis of $J_{p}$, 
then $p$ is lucky for $J$ in the sense of Arnold if $\LM(H(0))=\LM(H(p))$. 
It is shown in \cite[Thm.\ 5.12 and 6.2]{ArnoldE} that if $p$ is lucky for $J$ in this sense,
then $H(0)_{p}$ is well--defined and equal to $H(p)$. By 
\cite[Cor.\ 5.4 and Thm.\ 5.13]{ArnoldE},  all but finitely many primes 
are Arnold--lucky for $J$.
Moreover, if $\mathcal{P}$ is a set of primes satisfying Arnold's 
condition $\LM(H(0))=\LM(H(p))$ for all $p\in \mathcal{P}$, and such
that $\mathcal{P}$ is sufficiently large with respect to the coefficients
occurring in $H(0)$, then the $H(p)$, $p\in \mathcal{P}$, lift to $H(0)$.

In our situation, if $p$ is a prime number, we find $U(p)$ on our way, 
but $U(0)_{p}$ is only known to us after $U(0)$ has been computed.  
Therefore, the condition $U(0)_{p}=U(p)$ in our definition of lucky can 
only be checked a posteriori. Similarly for $V(0)_{p}=V(p)$. However, 
when performing our modular algorithm, by part 1 of Lemma 
\ref{lem lucky lift} below and Remark \ref{rmk error tolerant farey},
there are only finitely many primes not satisfying these conditions
and these primes will not influence the result of the algorithm.
\end{rem}

\begin{lem}
\label{lem lucky lift}With notation as above, we have:

\begin{enumerate}
\item All but a finite number of primes are lucky for $A$.

\item If $\mathcal{P}$ is a sufficiently large set of lucky primes for
$A$, then the reduced Gr\"{o}bner bases $G(p)$, $p\in\mathcal{P}$, lift to the
reduced Gr\"{o}bner basis $G(0)$. In the same way, the $R(p)$, $p\in
\mathcal{P}$, lift to $R(0)$.
\end{enumerate}
\end{lem}

\begin{proof}
With respect to part 1, it is clear that conditions (1a) and (1b) in our
definition of lucky are true for all but finitely many primes. Moreover,
$\frac{1}{d_p}U(0)_p$ is integral over $A_p$ for all but finitely many $p$.
Since testing normality via the Grauert and Remmert criterion amounts 
to a Gr\"{o}bner basis computation, and since reducing a Gr\"{o}bner 
basis modulo a sufficiently general prime $p$ gives a 
Gr\"{o}bner basis of the reduced ideal, we conclude
that $U(0)_{p}=U(p)$ for all but finitely many primes.
Furthermore, if  $U(0)_{p}=U(p)$, condition (1c) from our definition of 
lucky is equivalent to asking that $p$ is lucky for $U(0)$ 
in the sense of Arnold, so that also this condition  holds 
for all but finitely many primes. For $V(0)_{p}=V(p)$ and condition
(1d), we may argue similarly since finding the ideal of algebra relations 
amounts to another Gr\"obner basis computation.

For part 2, let $\mathcal{P}$ be a sufficiently large set of lucky
primes for $A$. Then, as pointed out above, $G(0)_{p}$ is
well--defined and equal to $G(p)$ for all $p\in \mathcal{P}$.
Furthermore, since $\mathcal{P}$ is sufficiently large, the
$G(0)_{p}$, $p\in\mathcal{P}$, lift to $G(0)$. In the same way,
we may argue for the relations.
\end{proof}

From a theoretical point of view, the idea of the algorithm is now as follows:
Consider a sufficiently large set $\mathcal{P}$ of lucky primes for $A$,
compute the reduced Gr\"{o}bner bases $G(p)$, $p\in\mathcal{P}$, and
lift the results to the reduced Gr\"{o}bner basis $G(0)$
as described above.

From a practical point of view, we face the problem that the conditions
(1c), (1d), and (2) from Definition
\ref{defnLucky} cannot be tested a priori. To remedy the situation, 
we proceed in a randomized way. First, we fix an integer
$t\geq1$ and choose a set of $t$ primes $\mathcal{P}$ at random. 
Second, we delete all primes $p$ from $\mathcal{P}$ which do not 
satisfy conditions (1a) and (1b). Third, we compute 
$\mathcal{GP}=\{G(p)\mid p\in\mathcal{P}\}$ and
$\mathcal{RP}=\{R(p)\mid p\in\mathcal{P}\}$,
and use the following test to modify $\mathcal{P}$ so that all primes in 
$\mathcal{P}$ satisfy (1c) and (1d) with high probability:

\vspace{0.2cm}

\noindent\emph{\textsc{deleteUnluckyPrimesNormal:} Define an equivalence
relation on $\mathcal{P}$ by setting $p\sim q:\Longleftrightarrow
\big(\LM(G(p))=\LM(G(q))$ and $\LM(R(p))=\LM(R(q))\big)$. Then replace
$\mathcal{P}$ by an equivalence class of largest\footnote{If 
applicable, take Remark \ref{rem:del-lucky-primes-repeated} below 
into account.} cardinality, and change
$\mathcal{GP}$ and $\mathcal{RP}$ accordingly.} 

\vspace{0.2cm}

Only now, we lift the Gr\"{o}bner bases in $\mathcal{GP}$ and $\mathcal{RP}$ 
to sets of  polynomials $G$ and $R$, respectively. Since we do not know 
whether all primes in the chosen equivalence class are indeed lucky and
whether the class is sufficiently large, a final verification step is needed:
We have to check whether $\frac{1}{d}\langle G\rangle$ is integral over 
$A$ and normal. Since this can be expensive, especially if the result is 
false, we test the result at first in positive characteristic:\vspace{0.2cm}

\noindent\emph{\textsc{pTestNormal:} Randomly choose a prime number 
$p\notin\mathcal{P}$ such that $A_{p}$ is a domain, $d_{p}$ is a non--zero 
element in the conductor of $A_{p}$, and $p$ does not divide the numerator 
and denominator of any coefficient occurring in a polynomial in $G$, $R$,
or $\left\{  {f_{1}},\ldots,{f_{r}}\right\}$. Return
true if $\frac{1}{d_{p}}\left\langle G_{p}\right\rangle $ is the normalization
of $A_{p}$ and satisfies the relations $R_{p}$, and  false
otherwise.}

\vspace{0.2cm}

If \textsc{pTestNormal} returns false, then $\mathcal{P}$ is not sufficiently
large for $A$ or not all primes in $\mathcal{P}$ are lucky (or the extra
prime chosen in \textsc{pTestNormal} is unlucky).
In this case, we enlarge the set $\mathcal{P}$ by $t$ primes 
not used so far and repeat the whole process.
On the other hand, if \textsc{pTestNormal} returns true, then most likely
$G=G(0)$ and, thus, $\frac{1}{d}\left\langle G\right\rangle = \overline{A}$. 
It makes, then, sense to verify the result over the rationals by applying
the following lemma. If the verification fails, we enlarge $\mathcal{P}$ and 
repeat the process.

\begin{lem}
With notation as above, the ring $\frac{1}{d}\left\langle G\right\rangle 
\subseteq\Quot(A)$ is the normalization of $A$ if and only if the
following two conditions hold:

\begin{enumerate}
\item The ring $\frac{1}{d}\left\langle G\right\rangle $ is integral over $A$.
This holds if $G$ and $R$ are Gr\"{o}bner bases, and the elements of
$\frac{1}{d}G $ satisfy the relations $R$.

\item The ring $\frac{1}{d}\left\langle G\right\rangle $ is normal.
Equivalently, $\frac{1}{d}\left\langle G\right\rangle $ satisfies the 
conditions of the Grauert and Remmert criterion.
\end{enumerate}
\end{lem}

\begin{proof}
If $\frac{1}{d}\left\langle G\right\rangle$ is integral over $A$, then
$\frac 1d \gen G \subseteq\overline{A}$.
If $\frac{1}{d}\left\langle G\right\rangle$ is also normal, then equality holds.
Note that if $R$ is a Gr\"{o}bner
basis, then $\dim\left\langle R\right\rangle =\dim\left\langle
R(p)\right\rangle $ for all $p\in\mathcal{P}$. Hence, if the elements of $\frac{1}%
{d}G $ satisfy the relations $R$, and $G$ is a Gr\"obner basis, then
$\frac{1}{d}\left\langle G\right\rangle$ is integral over $A$.
\end{proof}

We summarize modular normalization in Algorithm \ref{algModNormal}.%

\begin{algorithm}[h]
\caption{Modular normalization} \label{algModNormal}
\begin{algorithmic}[1]
\REQUIRE A prime ideal $I \subseteq \Q[X]$.
\ENSURE A Gr\"obner basis $G \subseteq \Q[X]$ and $d \in \Q[X]$ such that
$\frac 1d \gen G \subseteq\Quot(A)$ is the
normalization of $A = \Q[X]/I$.
\vspace{0.1cm}
\STATE compute $M$, the Jacobian ideal of $A$;
\STATE choose a polynomial $d \in \Q[X]$ representing a non--zero element $d\in M$;
\STATE choose $\mathcal P$, a list of random primes;
\STATE $\mathcal{GP} = \emptyset$, $\mathcal{RP} = \emptyset$;
\LOOP
\FOR{$p \in \mathcal P$}
\IF{$A_p$ is not a domain or $d_p\in A_p$ is zero 
   or $d_p$ is not contained in the conductor of $A_p$}
\STATE delete $p$;
\ELSE
\STATE use the GLS algorithm to compute $G(p)$, the reduced Gr\"obner basis such that
$\frac{1}{d_p} \gen{G(p)} \subseteq \Quot(A_p)$ is
the normalization of $A_p$, and $R(p)$, the reduced Gr\"obner basis of the ideal of algebra relations;
\STATE $\mathcal{GP} = \mathcal{GP} \cup \{G(p)\}$, $\mathcal{RP} = \mathcal{RP} \cup \{R(p)\}$;
\ENDIF
\ENDFOR
\STATE $(\mathcal{GP},\mathcal{RP},\mathcal{P}) = \textsc{deleteUnluckyPrimesNormal}(\mathcal{GP},
\mathcal{RP},\mathcal{P})$;
\STATE lift $(\mathcal{GP},\mathcal{RP},\mathcal{P})$ to $G \subseteq \Q[X]$ and $R \subseteq 
W[T_{1},\dots,T_{s}]$ via Chinese remaindering
and the Farey rational map;
\IF{the lift succeeds and \textsc{pTestNormal}$(I,d,G,R,\mathcal P)$}
\IF{$\frac 1d \gen G \subseteq \Quot(A)$ is integral over $A$}
\IF{$\frac 1d \gen G \subseteq \Quot(A)$ is normal}
\RETURN $(G,d)$;
\ENDIF
\ENDIF
\ENDIF
\STATE enlarge $\mathcal P$;
\ENDLOOP
\end{algorithmic}
\end{algorithm}

\begin{rem}\label{rem:del-lucky-primes-repeated}
If the loop in Algorithm \ref{algModNormal} requires more than one round, 
we have to apply \textsc{deleteUnluckyPrimesNormal} in Step 12 with some
care. Otherwise, it may happen that always classes containing only unlucky primes are selected.
To avoid this problem, when determining the 
cardinality of the classes considered in a certain round of the loop, we 
count all prime numbers in the class selected in the previous round 
as just one element.
Then $\mathcal{P}$ will eventually contain lucky primes and
termination of the algorithm is ensured by Lemma \ref{lem lucky lift}
and Remark \ref{rmk error tolerant farey}.
\end{rem}

\begin{rem}
In Algorithm \ref{algModNormal}, the normalizations $\frac{1}{d_{p}%
}\left\langle G(p)\right\rangle $ can be computed in parallel. Furthermore, we
can parallelize the final verifications of integrality and normality.
\end{rem}

\begin{rem}
Algorithm \ref{algModNormal} is also applicable without the final tests (that
is, without the verification that $\frac{1}{d}\left\langle G\right\rangle
\subseteq\Quot(A)$ is integral over $A$ and normal). In this case, the
algorithm is probabilistic, that is, the output $\frac{1}{d}\left\langle
G\right\rangle \subseteq\Quot(A)$ is the normalization of $A$ only with high
probability. This usually accelerates the algorithm considerably.
\end{rem}

\begin{rem}
\label{remark:no relations} The computation of the algebra structure $R$ of the
normalization via lifting of the relations $R(p)$ may require a large number
of primes. Hence, if the number of cores available is limited, a better choice
is to obtain just $G$ by the modular approach and then compute the relations
$R(0)$ over the rationals. For this approach, the initial ideals of the relations
need not be tested in \emph{\textsc{deleteUnluckyPrimesNormal}} and 
\textsc{pTestNormal}.
\end{rem}

\section{Timings\label{sect:timings}}

We compare the GLS normalization algorithm\footnote{We use the implementation
available in the \textsc{Singular} library \texttt{normal.lib}.} (denoted in
the tables below by \texttt{normal}) with Algorithm \ref{alg2} from Section
\ref{sect:norm-via-loc} (\texttt{locNormal}) and Algorithm \ref{algModNormal}
from Section \ref{secModMeth} (\texttt{modNormal})\footnote{To implement our
algorithms, we have created the \textsc{Singular} libraries
\texttt{modnormal.lib} and \texttt{locnormal.lib}.}. For all modular
computations, we use the simplified algorithm as specified in Remark
\ref{remark:no relations}. Note that at this writing, modularized versions of
\texttt{locNormal} have not yet been implemented.

In many cases, it turns out that the final verification is a time consuming
step of \texttt{modNormal}. To quantify the improvement of computation times
by omitting the verification, we give timings for the resulting, now
probabilistic, version of Algorithm \ref{algModNormal} (denoted by
\texttt{modNormal}$^{\mathtt{\ast}}$ in the tables). In all examples computed
so far, the result of the probabilistic algorithm is indeed correct.

All timings are in seconds on an AMD Opteron 6174 machine with 48 cores,
2.2 GHz, and 128 GB of RAM, running a Linux operating system. Computations which
did not finish within 2000 seconds are marked by a dash. The maximum number of
cores used is written in square brackets. For the single core version of
\texttt{modNormal}, we indicate the number of primes used by the algorithm in brackets.

So far, in \textsc{Singular}, the computation of associated primes via fast
modular methods has only been implemented for the zero--dimensional case. As
the computation of associated primes is required by the local approach, we
first focus on the case of curves, where the singular locus is
zero--dimensional. 

\vspace{0.2cm}

The projective plane curves defined by the equations%
\[
f_{1,k}=\left(  X^{k+1}+Y^{k+1}+Z^{k+1}\right)  ^{2}-4\left(  X^{k+1}%
Y^{k+1}+Y^{k+1}Z^{k+1}+Z^{k+1}X^{k+1}\right)
\]
were constructed in \cite{H}. They have $3\left(  k+1\right)  $ singularities
of type $A_{k}$, provided that $k$ is even. If $k$ is odd, the curves are
reducible, in which case the normalization algorithms still work in the same way 
as in the irreducible case as long as
they do not detect a zerodivisor. After the coordinate transformation
$Z\mapsto3X-2Y+Z$, all singularities of the projective curves lie in the
affine chart $Z\neq0$. We apply the algorithm to the affine curves. The
timings for $k=2,...,5$ are shown in Table \ref{table timings1}.

\begin{table}[ht]
\centering
\caption{Timings for plane curves with many $A_k$ singularities:}
\begin{tabular}
[c]{l|cc|cc|cc|c}
& $f_{1,2}$ &  & $f_{1,3}$ &  & $f_{1,4}$ &  & $f_{1,5}$\\\hline
\texttt{normal[1]} & $.34$ &  & $14$ &  & $-$ &  & $-$\\
\texttt{locNormal[1]} & $.57$ &  & $2.0$ &  & $2.1$ &  & $38$\\
\texttt{locNormal[20]} & $.42$ &  & $1.3$ &  & $1.4$ &  & $11$\\
\texttt{modNormal[1]} & $4.4$ & \hspace{-0.25in}$(3)$\hspace{-0.15in} & $73$ &
\hspace{-0.25in}$(4)$\hspace{-0.15in} & $250$ & \hspace{-0.25in}$(5)$%
\hspace{-0.15in} & $-$\\
\texttt{modNormal[10]} & $4.1$ &  & $68$ &  & $240$ &  & $-$\\
\texttt{modNormal$^{\mathtt{\ast}}$[1]} & $.57$ & \hspace{-0.25in}$(3)$%
\hspace{-0.15in} & $7.4$ & \hspace{-0.25in}$(4)$\hspace{-0.15in} & $11$ &
\hspace{-0.25in}$(5)$\hspace{-0.15in} & $-$\\
\texttt{modNormal$^{\mathtt{\ast}}$[10]} & $.31$ &  & $2.1$ &  & $2.5$ &  &
$-$%
\end{tabular}
\label{table timings1}
\end{table}

Both the local and the probabilistic modular approach have a better
performance than the GLS algorithm, and they improve further in their parallel
versions. The modular algorithm with final verification is slower, but can
still handle much bigger examples than GLS.

\vspace{0.2cm}

Timings for the affine plane curves defined by
{\footnotesize
\begin{align*}
f_{2,k}  &  =((X-1)^{k}-Y^{3})((X+1)^{k}-Y^{3})(X^{k}-Y^{3})((X-2)^{k}%
-Y^{3})((X+2)^{k}-Y^{3})+Y^{15},\\
f_{3}  &  =X^{10}+Y^{10}+(X-2Y+1)^{10}+2(X^{5}(X-2Y+1)^{5}-X^{5}Y^{5}%
+Y^{5}(X-2Y+1)^{5}),\\
f_{4}  &  =(Y^{5}+2X^{8})(Y^{3}+7(X-1)^{4})((Y+5)^{7}+2X^{12})+Y^{11},\\
f_{5}  &  =9127158539954X^{10}+3212722859346X^{8}Y^{2}+228715574724X^{6}%
Y^{4}\\
&  -34263110700X^{4}Y^{6}-5431439286X^{2}Y^{8}-201803238Y^{10}%
-134266087241X^{8}\\
&  -15052058268X^{6}Y^{2}+12024807786X^{4}Y^{4}+506101284X^{2}Y^{6}%
-202172841Y^{8}\\
&  +761328152X^{6}-128361096X^{4}Y^{2}+47970216X^{2}Y^{4}-6697080Y^{6}\\
&  -2042158X^{4}+660492X^{2}Y^{2}-84366Y^{4}+2494X^{2}-474Y^{2}-1,
\end{align*}} 
are presented in Table \ref{table timings2}.

\begin{table}[ht]
\centering
\caption{Timings for plane curves with various types of singularities:}
\begin{tabular}
[c]{l|cc|cc|cc|cc|cc|cc}
& $f_{2,7}$ &  & $f_{2,8}$ &  & $f_{2,9}$ &  & $f_{3}$ &  & $f_{4}$ &  &
$f_{5}$ & \\\hline
\texttt{normal[1]} & $7.7$ &  & $12$ &  & $383$ &  & $-$ &  & $474$ &  &
$1620$ & \\
\texttt{locNormal[1]} & $4.4$ &  & $13$ &  & $118$ &  & $1.9$ &  & $19$ &  &
$1.2$ & \\
\texttt{locNormal[20]} & $1.4$ &  & $3.3$ &  & $31$ &  & $1.4$ &  & $18$ &  &
$.93$ & \\
\texttt{modNormal[1]} & $38$ & \hspace{-0.25in}$(3)$\hspace{-0.15in} & $69$ &
\hspace{-0.25in}$(3)$\hspace{-0.15in} & $146$ & \hspace{-0.25in}$(3)$%
\hspace{-0.15in} & $142$ & \hspace{-0.25in}$(3)$\hspace{-0.15in} & $-$ &  &
$50$ & \hspace{-0.15in}$(8)$\\
\texttt{modNormal[10]} & $38$ &  & $69$ &  & $146$ &  & $84$ &  & $-$ &  &
$43$ & \\
\texttt{modNormal$^{\mathtt{\ast}}$[1]} & $.70$ & \hspace{-0.25in}$(3)$%
\hspace{-0.15in} & $1.2$ & \hspace{-0.25in}$(3)$\hspace{-0.15in} & $1.2$ &
\hspace{-0.25in}$(3)$\hspace{-0.15in} & $88$ & \hspace{-0.25in}$(3)$%
\hspace{-0.15in} & $9.8$ & \hspace{-0.25in}$(3)$\hspace{-0.15in} & $7.0$ &
\hspace{-0.15in}$(8)$\\
\texttt{modNormal$^{\mathtt{\ast}}$[10]} & $.47$ &  & $.70$ &  & $.74$ &  &
$30$ &  & $4.7$ &  & $.98$ &
\end{tabular}
\label{table timings2}
\end{table}

\vspace{0.2cm}

In Table \ref{table timings3}, we consider surfaces in $\mathbb{A}^{3}$ cut out by
{\footnotesize
\begin{align*}
f_{6,k}  &  =XY(X-Y)(X+Y)(Y-1)Z+(X^{k}-Y^{2})(X^{10}-(Y-1)^{2}),\\
f_{7,k}  &  =Z^{2}-(Y^{2}-1234X^{3})^{k}(15791X^{2}-Y^{3})(1231Y^{2}%
-X^{2}(X+158))(1357Y^{5}-3X^{11}),\\
f_{8}  &  =Z^{5}-((13X-17Y)(5X^{2}-7Y^{3})(3X^{3}-2Y^{2})(19Y^{2}%
-23X^{2}(X+29)))^{2}.
\end{align*}} 

We omit the verification step in the modular algorithm, as this is too time
consuming.

\begin{table}[ht]
\centering
\caption{Timings for the normalization of surfaces in $\mathbb{A}^{3}$:}
\begin{tabular}
[c]{l|cc|cc|cc|cc|cc|cc}
& $f_{6,11}$ &  & $f_{6,12}$ &  & $f_{6,13}$ &  & $f_{7,2}$ &  & $f_{7,3}$ &
& $f_{8}$ & \\\hline
\texttt{normal[1]} & $2.6$ &  & $11$ &  & $6.4$ &  & $-$ &  & $-$ &  & $-$ &
\\
\texttt{locNormal[1]} & $.25$ &  & $.26$ &  & $.29$ &  & $80$ &  & $113$ &  &
$70$ & \\
\texttt{locNormal[20]} & $.21$ &  & $.22$ &  & $.24$ &  & $80$ &  & $113$ &  &
$70$ & \\
\texttt{modNormal$^{\mathtt{\ast}}$[1]} & $2.2$ & \hspace{-0.25in}$(2)$%
\hspace{-0.15in} & $.60$ & \hspace{-0.25in}$(2)$\hspace{-0.15in} & $.78$ &
\hspace{-0.25in}$(2)$\hspace{-0.15in} & $12$ & \hspace{-0.25in}$(5)$%
\hspace{-0.15in} & $17$ & \hspace{-0.25in}$(5)$\hspace{-0.15in} & $2.3$ &
\hspace{-0.15in}$(2)$\\
\texttt{modNormal$^{\mathtt{\ast}}$[10]} & $1.5$ &  & $.52$ &  & $.67$ &  &
$3.5$ &  & $4.7$ &  & $1.7$ &
\end{tabular}
\label{table timings3}
\end{table}

Note that the performance of the local approach will be considerably improved
as soon as modular primary decomposition in higher dimension will be available
in \textsc{Singular}.

\vspace{0.2cm}

Timings for the curves in $\mathbb{A}^{3}$ defined by the ideals 
\[
I_{9,k}=\left\langle Z^{3}-(19Y^{2}-23X^{2}(X+29))^{2},X^{3}-(11Y^{2}%
-13Z^{2}(Z+1))^{k}\right\rangle
\]
and the surface in $\mathbb{A}^{4}$ defined by 
\begin{align*}
I_{10}=\big\langle &Z^{2}-(Y^{3}-123456W^{2})(15791X^{2}-Y^{3})^{2}, \\
  & WZ-(1231Y^{2}-X(111X+158))\big\rangle
\end{align*}
are given in Table \ref{table timings4}.

\begin{table}[h]
\centering
\caption{Timings for curves in $\mathbb{A}^{3}$ and a surface in $\mathbb{A}^{4}$:}
\begin{tabular}
[c]{l|cc|cc|cc}
& $I_{9,1}$ &  & $I_{9,2}$ &  & $I_{10}$ & \\\hline
\texttt{normal[1]} & $3.2$ &  & $-$ &  & $150$ & \\
\texttt{locNormal[1]} & $4.2$ &  & $36$ &  & $83$ & \\
\texttt{locNormal[20]} & $4.1$ &  & $35$ &  & $82$ & \\
\texttt{modNormal[1]} & $-$ &  & $-$ &  & $28$ & \hspace{-0.15in}$(4)$\\
\texttt{modNormal[10]} & $-$ &  & $-$ &  & $14$ & \\
\texttt{modNormal$^{\mathtt{\ast}}$[1]} & $8.9$ & \hspace{-0.25in}$(5)$%
\hspace{-0.15in} & $-$ &  & $8.4$ & \hspace{-0.15in}$(4)$\\
\texttt{modNormal$^{\mathtt{\ast}}$[10]} & $2.1$ &  & $-$ &  & $2.5$ &
\end{tabular}
\label{table timings4}
\end{table}

To summarize, both the local and the probabilistic modular approach provide a
significant improvement over the GLS algorithm in computation times and size
of the examples covered. The probabilistic method is very stable in the sense
that it produces the correct result in all examples computed so far. As usual,
the verification step in the modular setup is the most time consuming task,
and a refinement of this step will be the focus of further research. The
modular technique parallelizes completely, the local approach parallelizes
best if the complexity distributes evenly over the minimal strata of the
singular locus. In general, the localization technique, even when not run in
parallel, is a major improvement to the GLS algorithm. Note, that the local
contribution can also be obtained by other means. See, for example,
\cite{BDLS} for a fast method in the case of curves, using Hensel lifting and
Puiseux series.

\end{document}